\documentclass[reqno,10pt]{amsart}

\usepackage{amsmath}
\usepackage{amssymb}
\usepackage{amsthm}
\usepackage{xspace}
\usepackage{pictexwd}
\usepackage{float}  \restylefloat{figure} %\restylefloat{table}
\usepackage[OT2,OT1]{fontenc}
\usepackage{amscd}
\usepackage{enumerate}

\theoremstyle{plain}
\newtheorem{thm}{Theorem}[section]
\newtheorem*{thm*}{Theorem}
\newtheorem{prop}[thm]{Proposition}
\newtheorem{lem}[thm]{Lemma}

\theoremstyle{definition}
\newtheorem{mydef}{Definition}[section]

\theoremstyle{remark}

\let\myskip=\medskip

\def\definebb#1=#2.{\def#1{{{\mathbb #2}^{\vphantom{x}}}}}
\definebb \c=C.
\definebb \r=R.
\definebb \s=S.
\definebb \z=Z.
\definebb \h=H.
\def\Cal#1{\mathcal{#1}}

\def\calB{{\Cal B}}
\def\calM{{\Cal M}}
\def\calV{{\Cal V}}

\def\hyp{\h}
\def\arf{\si}

\def\piorb{\pi_1}
\def\piorbO{\pi_1^0}

\def\tT{\tilde{T}}

\def\al{{\alpha}}
\def\be{{\beta}}
\def\Ga{{\Gamma}}
\def\ga{{\gamma}}
\def\de{{\delta}}

\def\ve{{\varepsilon}}

\def\si{{\sigma}}
\def\Si{{\Sigma}}

\def\Om{{\Omega}}

\let\emptyset=\varnothing
\def\st{\,\,\big|\,\,}
\def\<{\langle}
\def\>{\rangle}

\let\ge=\geqslant
\let\le=\leqslant

\DeclareMathOperator{\Aut}{Aut}

\DeclareMathOperator{\Mod}{Mod}

\let\mod=\undefined \DeclareMathOperator{\mod}{mod}

\begin{document}

\author[Sergey Natanzon]{Sergey Natanzon}
\address{National Research University Higher School of Economics, Vavilova Street 7, 117312 Moscow, Russia}
%\address{Laboratory of Quantum Topology, Chelyabinsk State University, Chelyabinsk, Russia}
%\address{Belozersky Institute of Phys.-Chem. Biology, Moscow State University, Korp.~A, Leninske Gory, 11899 Moscow, Russia}
\address{Institute of Theoretical and Experimental Physics (ITEP), Moscow, Russia}
%\address{Independent University of Moscow, Bolshoi Vlasevsky Pereulok, 11 Moscow, Russia}
\email{natanzon@mccme.ru}
\author[Anna Pratoussevitch]{Anna Pratoussevitch}
\address{Department of Mathematical Sciences\\ University of Liverpool\\ Liverpool L69~7ZL}
\email{annap@liv.ac.uk}

\title{Moduli Spaces of Higher Spin Klein Surfaces}

%\title{Higher Spin Bundles on Klein Surfaces}
%\title{Real Forms of Higher Spin Riemann Surfaces}

\begin{date}  {\today} \end{date}

\thanks{Grant support for S.N.: The article was prepared within the framework of the Academic Fund Program
at the National Research University Higher School of Economics (HSE) in 2015--16 (grant Nr 15-01-0052)
and supported within the framework of a subsidy granted to the HSE by the Government of the Russian Federation for the implementation of the Global Competitiveness Program.
Grant support for A.P.: The work was supported in part by the Leverhulme Trust grant RPG-057.}

%S.N. was partially supported by the Laboratory of Quantum Topology of Chelyabinsk State University
% (Russian Federation Government grant 14.Z50.31.0020)

% S.N.:grants RFBR 13-01-00755, NSh 5138.2014.1

\begin{abstract}
We study the connected components of the space of higher spin bundles on hyperbolic Klein surfaces.
A Klein surface is a generalisation of a Riemann surface
to the case of non-orientable surfaces or surfaces with boundary.
The category of Klein surfaces is isomorphic to the category of real algebraic curves.
An $m$-spin bundle on a Klein surface is a complex line bundle whose $m$-th tensor power is the cotangent bundle.
The spaces of higher spin bundles on Klein surfaces are important because of their applications in singularity theory
and real algebraic geometry,
in particular for the study of real forms of Gorenstein quasi-homogeneous surface singularities.
In this paper we describe all connected components of the space of higher spin bundles on hyperbolic Klein surfaces
in terms of their topological invariants and prove that any connected component
is homeomorphic to a quotient of ${\mathbb R}^d$ by a discrete group.
We also discuss applications to real forms of Brieskorn-Pham singularities.
\end{abstract}


% 30F50: Klein surfaces

% 14H60: Vector bundles on curves (and their moduli?)

% 30F35: Fuchsian groups and automorphic functions

% 30F60: Teichmueller theory


\maketitle

%\tableofcontents

\section{Introduction}

\myskip
A complex line bundle $e:L\to P$ on a Riemann surface~$P$, denoted $(e,P)$,
is an $m$-{\it spin bundle\/} for an integer $m>1$ if its $m$-th tensor power $e^{\otimes m}:L^{\otimes m}\to P$
is isomorphic to the cotangent bundle of~$P$.
The classical $2$-spin structures on compact Riemann surfaces
%of genus $g=g(P)$
were introduced by Riemann as theta characteristics
%~\cite{R}.
and play an important role in mathematics.
Their modern interpretation as complex line bundles and classification was given by Atiyah~\cite{A} and Mumford~\cite{Mu},
who showed that $2$-spin bundles have a topological invariant $\de=\de(e,P)$ in~$\{0,1\}$, the {\it Arf invariant},
which is determined by the parity of the dimension of the space of sections of the bundle.
Moreover, the space $S^2_{g,\de}$ of $2$-spin bundles on Riemann surfaces of genus~$g$ with Arf invariant~$\de$,
i.e.\ the space of such pairs $(e,P)$,
is homeomorphic to a quotient of $\r^{6g-6}$ by a discrete group of autohomeomorphisms, see~\cite{N1989a, Nbook}.

\myskip
The study of spaces of $m$-spin bundles for arbitrary~$m$ started more recently
because of the remarkable connections between the compactified moduli space of $m$-spin bundles
and the theory of integrable systems~\cite{Wi},
and because of their applications in singularity theory~\cite{Dolgachev:1983, NP:2011, NP:2013}.
It was shown that for odd~$m$ the space of $m$-spin bundles is connected,
while for even~$m$ (and $g>1$) there are two connected components,
distinguished by an invariant which generalises the Arf invariant~\cite{Jarvis:2000}.
In all cases each connected components of the space of $m$-spin bundles on Riemann surfaces of genus~$g$
is homeomorphic to a quotient of $\r^{6g-6}$ by a discrete group of autohomeomorphisms, see~\cite{NP:2005, NP:2009}.
The homology of these moduli spaces was studied further in~\cite{Jarvis:2001, JKV:2001, ChZ, FShZ, RW1, RW2, PPZ, SSZ}.

\myskip
The aim of this paper is to determine the topological structure of the space of $m$-spin bundles
on hyperbolic Klein surfaces.
A {\it Klein surface\/} is a non-orientable topological surface with a maximal atlas whose transition maps
are {\it dianalytic}, i.e.\ either holomorphic or anti-holomorphic, see~\cite{AllingGreenleaf:1971}.
Klein surfaces can be described as quotients $P/\<\tau\>$,
where $P$ is a compact Riemann surface and $\tau:P\to P$ is an anti-holomorphic involution on~$P$.
The category of such pairs is isomorphic to the category of Klein surfaces via $(P,\tau)\mapsto P/\<\tau\>$.
Under this correspondence the fixed points of~$\tau$ correspond to the boundary points of the Klein surface.
In this paper a Klein surface will be understood as an isomorphism class of such pairs $(P,\tau)$.
We will only consider connected compact Klein surfaces.
The category of connected compact Klein surfaces
is isomorphic to the category of irreducible real algebraic curves (see~\cite{AllingGreenleaf:1971}).
%Under this correspondence the fixed points of~$\tau$ correspond to the boundary points of the Klein surface
%and to the real points of the real algebraic curve.

\myskip
The boundary of the surface $P/\<\tau\>$, if not empty, decomposes into $k$ pairwise disjoint simple closed smooth curves.
These closed curves are called {\it ovals\/}
and correspond to connected components of the set of fixed points~$P^{\tau}$ of the involution $\tau:P\to P$.
On the real algebraic curve they correspond to connected components of the set of real points.

\myskip
The {\it topological type\/} of the surface $P/\<\tau\>$ is determined by the triple $(g,k,\ve)$,
where $g$ is the genus of~$P$,
$k$ is the number of connected components of the boundary of~$P/\<\tau\>$
and $\ve\in\{0,1\}$ with $\ve=1$ if the surface is orientable and $\ve=0$ otherwise.
The following conditions are satisfied:
$1\le k\le g+1$ and $k\equiv g+1~(\mod2)$ in the case $\ve=1$ and $0\le k\le g$ in the case $\ve=0$.
These classification results were obtained by Weichold~\cite{Weichold:1883}.
It is known that the topological type completely determines the connected component of the space of Klein surfaces.
Moreover, the space $M_{g,k,\ve}$ of Klein surfaces of topological type $(g,k,\ve)$
is homeomorphic to the quotient of $\r^{3g-3}$ by a discrete subgroup of automorphisms.
In addition to the invariants $(g,k,\ve)$,
it is useful to consider an invariant that we will call the geometric genus of~$(P,\tau)$.
In the case $\ve=1$ the geometric genus $(g+1-k)/2$ is the number of handles
that need to be attached to a sphere with holes to obtain a surface homeomorphic to~$P/\<\tau\>$.
In the case $\ve=0$ the geometric genus $[(g-k)/2]$ is the half of the number of M\"obius bands
that need to be attached to a sphere with holes to obtain a surface homeomorphic to~$P/\<\tau\>$.

\myskip
An {\it $m$-spin bundle on a Klein surface\/}~$(P,\tau)$ is a pair~$(e,\be)$,
where $e:L\to P$ is an $m$-spin bundle on~$P$ and $\be:L\to L$ is an anti-holomorphic involution on~$L$
such that $e\circ\be=\tau\circ e$, i.e.\ the following diagram commutes:
% $\be(z,x)=(\tau(z),*)$.
$$
  \begin{CD}
   L@>{e}>>P\\
   @V{\be}VV@VV{\tau}V\\
   L@>{e}>>P\\
  \end{CD}
$$

\myskip
The spaces of higher spin bundles on Klein surfaces are important because of their applications in singularity theory
and real algebraic geometry.
We are particularly interested in the applications to the classification of real forms of complex singularities.
Any Brieskorn-Pham singularity, i.e.\ singularity of the form $x^a+y^b+z^c=0$, can be constructed
from an $m$-spin bundle on a Riemann surface~$P$ (roughly speaking by contracting the zero section of the bundle)
\cite{Milnor:1975, Neumann:1977}
and real forms of the singularity correspond to $m$-spin bundles on Klein surfaces~$(P,\tau)$.
More generally any hyperbolic Gorenstein quasi-homogeneous surface singularity can be constructed
from an $m$-spin bundle on a quotient of the form $\hyp/\Ga$, where $\Ga$ is a Fuchsian group, possibly with torsion,
see~\cite{Dolgachev:1975, Dolgachev:1977, Dolgachev:1983}.
An extension of the results of our paper to such $m$-spin bundles
will lead to a classification of real forms of hyperbolic Gorenstein quasi-homogeneous surface singularities.
The first results in this direction were obtained by H.~Riley in her Ph.D.\ thesis~\cite{Ril}.
Other classes of complex singularities for which real forms have been studied are simple singularities and cusp singularities~\cite{AC, W1, W2, W4}.
See section~\ref{sec-sing} for more details of applications to singularity theory.

\myskip
Another important connection is between $2$-spin bundles on Klein surfaces
and Abelian Yang-Mills theory on real tori~\cite{OT:2013} and possible generalisations to $m$-spin bundles.

\myskip
In this paper we determine the connected components of the space of $m$-spin bundles on Klein surfaces,
i.e.\ equivalence classes of $m$-spin bundles on Klein surfaces up to topological equivalence (Definition~\ref{def-topeq-Arf}).
We find the topological invariants that determine such an equivalence class
and determine all possible values of these invariants.
We also show that every equivalence class is a connected set homeomorphic to a quotient of $\r^n$
by a discrete group, where the dimension~$n$ and the group depend on the class.
The special case $m=2$ was studied in~\cite{N1989b, N1990, N1999, Nbook}.
%For $m=2$ these results were obtained in~\cite{N1989b, N1990, N1999, Nbook}.

\myskip
While $2$-spin bundles on a Riemann surface~$P$ can be described in terms of quadratic forms on $H_1(P,\z/2\z)$,
for higher spin bundles the situation is more complex.
The main innovation of our method is to assign to every $m$-spin bundle on a Klein surface $(P,\tau)$
a function on the set of simple closed curves in~$P$ with values in $\z/m\z$, called real $m$-Arf function~\cite{NP:2016}.
Thus the problem of topological classification of $m$-spin bundles on Klein surfaces
is reduced to topological classification of real $m$-Arf functions.
We introduce a complete set of topological invariants of real $m$-Arf functions. 
We then construct for any real $m$-Arf function~$\arf$ a canonical generating set,
i.e.\ a generating set of the fundamental group of~$P$
on which $\arf$ assumes values determined by the topological invariants.

\myskip
We will now explain the results in more detail.
Let $(P,\tau)$ be a Klein surface of type $(g,k,\ve)$.
In this paper we will consider hyperbolic Klein surfaces $(P,\tau)$,
i.e.\ we assume that the underlying Riemann surface~$P$ is hyperbolic, $g\ge2$.
We will also assume that the geometric genus of $(P,\tau)$ is positive,
i.e.\ $k\le g-2$ if $\ve=0$ and $k\le g-1$ if $\ve=1$.
%The other cases, i.e.\ when $P$ is a sphere or a torus, require different techniques.
%The other cases, i.e.\ when $P$ is a sphere or a torus, are easier, but require different techniques.

\myskip
%Let $m$ be odd.
%In this case we show that $g\equiv1~(\mod m)$.
We show that if $m$ is odd and there exists an $m$-spin bundle on the Klein surface~$(P,\tau)$ then $g\equiv1~(\mod m)$.
Moreover, assuming that $m$ is odd and $g\equiv1~(\mod m)$,
the space of $m$-spin bundles on Klein surfaces of type $(g,k,\ve)$ is not empty and is connected.

\myskip
Now let $m$ be even.
Consider an $m$-spin bundle~$e$ on the Klein surface~$(P,\tau)$.
A restriction of the bundle~$e$ gives a bundle on the ovals.
Let $K_0$ and $K_1$ be the sets of ovals on which the bundle is trivial and non-trivial respectively.
%Let $k_i=|K_i|$ for $i=0,1$.
We show that $|K_1|\cdot m/2\equiv1-g~(\mod m)$.
% and $k_1$ is even? Why?

\myskip
If $m$ is even and $\ve=0$,
the Arf invariant~$\de$ of the bundle~$e$ and the cardinalities $k_i=|K_i|$ for $i=0,1$
determine a (non-empty) connected component of the space of $m$-spin bundles
on Klein surfaces of type $(g,k_0+k_1,0)$ if and only if
$$k_1\cdot\frac{m}2\equiv1-g~(\mod m).$$

\myskip
If $m$ is even and $\ve=1$,
the bundle~$e$ determines a decomposition of the set of ovals in two disjoint sets, $K^0$ and $K^1$,
of {\it similar\/} ovals (for details see section~\ref{sec-topinv}).
The bundle~$e$ induces $m$-spin bundles on connected components of $P\backslash P^{\tau}$.
The Arf invariant~$\tilde\de$ of these induced bundles
does not depend on the choice of the connected component of $P\backslash P^{\tau}$.
This invariant~$\tilde\de$ and the cardinalities $k_i^j=|K_i\cap K^j|$ for $i,j\in\{0,1\}$
determine a connected component of the space of $m$-spin bundles
on Klein surfaces of type $(g,k_0^0+k_0^1+k_1^0+k_1^1,1)$ if and only if 
\begin{enumerate}[$\bullet$]
\item
If $g>k+1$ and $k^0_0+k^1_0\ne0$ then $\tilde\de=0$.
\item
If $g>k+1$ and $m\equiv0~(\mod4)$ then $\tilde\de=0$.
\item
If $g=k+1$ and $k^0_0+k^1_0\ne0$ then $\tilde\de=1$.
\item
If $g=k+1$ and $m\equiv0~(\mod4)$ then $\tilde\de=1$.
\item
If $g=k+1$ and $k^0_0+k^1_0=0$ and $m\equiv2~(\mod4)$ then $\tilde\de\in\{1,2\}$.
\item
$(k^0_1+k^1_1)\cdot m/2\equiv1-g~(\mod m)$.
\end{enumerate}

\myskip
We also show that every connected component of the space of $m$-spin bundles on Klein surfaces of genus~$g$
is homeomorphic to a quotient of $\r^{3g-3}$ by a discrete subgroup of automorphisms which depends on the component
(see Theorem~\ref{thm-moduli-spin}).

\myskip
The paper is organised as follows:

\myskip
In section~\ref{sec-classification} we recall the classification of real $m$-Arf functions from~\cite{NP:2016}.
We determine the topological invariants of real $m$-Arf functions in section~\ref{sec-toptypes}.
In section~\ref{sec-moduli} we use these topological invariants
to describe connected components of the space of $m$-spin bundles on Klein surfaces.
In section~\ref{sec-sing} we explain the connection between $m$-spin bundles on Klein surfaces
and real forms of complex singularities. 

%\myskip
%The XXX author is grateful to XXX, where part of this work was done, for its hospitality and support.
%We would like to thank Victor Goryunov, Anna Felikson and Oscar Randal-Williams
%for useful discussions related to this work.
%We would like to thank the referees for their valuable remarks and suggestions.

\section{Higher Spin Structures on Klein Surfaces}

\label{sec-classification}

\myskip
A {\it Klein surface} is a topological surface with a maximal atlas whose transition maps are either holomorphic or anti-holomorphic.
A {\it homomorphism\/} between Klein surfaces is a continuous mapping which is either holomorphic or anti-holomorphic in local charts.

\myskip
Let us consider pairs~$(P,\tau)$,
where $P$ is a compact Riemann surface and $\tau:P\to P$ is an anti-holomorphic involution on~$P$.
For each such pair~$(P,\tau)$ the quotient $P/\<\tau\>$ is a Klein surface
and each isomorphism class of Klein surfaces contains a surface of the form $P/\<\tau\>$.
Moreover, two such quotients $P_1/\<\tau_1\>$ and $P_2/\<\tau_2\>$ are isomorphic as Klein surfaces
if and only if there exists a biholomorphic map~$\psi:P_1\to P_2$ such that $\psi\circ\tau_1=\tau_2\circ\psi$,
in which case we say that the pairs $(P_1,\tau_1)$ and $(P_2,\tau_2)$ are {\it isomorphic}.
Hence from now on instead of Klein surfaces we will consider isomorphism classes of pairs $(P,\tau)$.
The category of such pairs $(P,\tau)$ is isomorphic to the category of real algebraic curves,
where fixed points of~$\tau$ (i.e.\ boundary points of the corresponding Klein surface)
correspond to real points of the real algebraic curve.
For example a non-singular plane real algebraic curve given by an equation $F(x,y)=0$ 
is the set of real points of such a pair~$(P,\tau)$,
where $P$ is the normalisation and compactification of the surface $\{(x,y)\in\c^2\st F(x,y)=0\}$
and $\tau$ is given by the complex conjugation, $\tau(x,y)=(\bar x,\bar y)$.

\myskip
Given two Klein surfaces $(P_1,\tau_1)$ and $(P_2,\tau_2)$,
we say that they are {\it topologically equivalent\/}
if there exists a homeomorhism~$\phi:P_1\to P_2$ such that $\phi\circ\tau_1=\tau_2\circ\phi$.

\myskip
Let $(P,\tau)$ be a Klein surface.
We say that~$(P,\tau)$ is {\it separating\/} if the set~$P\backslash P^{\tau}$ is not connected,
otherwise we say that it is {\it non-separating\/}.
The set of fixed points of the involution~$\tau$ is called the {\it set of real points\/} of~$(P,\tau)$
and denoted by~$P^{\tau}$.
The set~$P^{\tau}$ decomposes into pairwise disjoint simple closed smooth curves, called {\it ovals}.
Simple closed curves on~$P$ which are invariant under the involution~$\tau$
but do not contain any fixed points of~$\tau$ are called {\it twists}.
The {\it topological type\/} of~$(P,\tau)$ is the triple~$(g,k,\ve)$,
where $g$ is the genus of the Riemann surface~$P$,
$k$ is the number of connected components of the fixed point set $P^{\tau}$ of~$\tau$,
$\ve=0$ if $(P,\tau)$ is non-separating and $\ve=1$ otherwise.
In this paper we consider hyperbolic surfaces, hence $g\ge2$.
% Nbook, Corollary 1.1 on p.73
% NP:2016, Thm. 4.1
Weichold~\cite{Weichold:1883} classified Klein surfaces up to topological equivalence:
Two Klein surfaces are topologically equivalent if and only if they are of the same topological type.
A triple $(g,k,\ve)$ is a topological type of some Klein surface if and only if
either $\ve=1$, $1\le k\le g+1$, $k\equiv g+1~(\mod2)$ or $\ve=0$, $0\le k\le g$.
For more detailed discussion of Klein surfaces see~\cite{AllingGreenleaf:1971,N1990}.

\myskip
%A hyperbolic Riemann surface~$P$ can be described as a quotient~$P=\hyp/\Ga$
%of the hyperbolic plane~$\hyp$ by the action of a Fuchsian group~$\Ga$.
A line bundle~$e:L\to P$ on a Riemann surface~$P$ is an $m$-spin bundle (of rank~$1$)
if the $m$-fold tensor power~$L\otimes\cdots\otimes L\to P$ coincides with the cotangent bundle of~$P$.
For $m=2$ we obtain the classical notion of a spin bundle.
%\label{thm-corresp-spin-arf}
In~\cite{NP:2005, NP:2009} we proved that $m$-spin bundles on~$P$ are in 1-1-correspondence
with $m$-{\it Arf functions\/},
certain functions on the space $\piorbO(P)$ of homotopy classes of simple closed curves on~$P$ with values in~$\z/m\z$ described by simple geometric properties.
We introduced topological invariants of $m$-Arf functions,
in particular the {\it Arf inariant\/}~$\de$,
and described the conditions for the existence of an $m$-Arf function with prescribed values on a generating set of~$\piorb(P)$.

\myskip
Let $(P,\tau)$ be a Klein surface.
A classification of $m$-spin bundles on~$P$ that are invariant under~$\tau$ was given in~\cite{NP:2016}.
%\label{corresp-bundles-arf}
% m=2 case: Nbook, Thm 5.1, p.87
% NP:2016, Thm 3.11
Such bundles are characterised by the special properties of the corresponding $m$-Arf functions,
called real $m$-Arf functions.
An $m$-Arf function~$\arf$ on~$P$ is {\it real\/} if $\arf(\tau c)=-\arf(c)$ for any~$c$ and $\arf(c)=0$ for any twist~$c$.
The mapping that assigns to an $m$-spin bundle on~$P$ the corresponding $m$-Arf function
establishes a 1-1-correspondence between $m$-spin bundles invariant under~$\tau$ and real $m$-Arf functions on $P$.
In~\cite{NP:2016} we determined the conditions for the existence of real $m$-Arf functions
with prescribed values on a {\it symmetric generating set\/,}
which is a generating set of $\piorb(P)$ which is particularly well adapted to the action of~$\tau$.
Furthermore we enumerated such real $m$-Arf functions.
For details see section 4{.}4 in~\cite{NP:2016}, in particular Theorems~4{.}9 and~4{.}10.

\section{Topological Types of Higher Arf Functions on Klein Surfaces}

\label{sec-toptypes}

\subsection{Topological Invariants}

\label{sec-topinv}

\begin{mydef}
\label{def-type-nonsep-even}
% Nbook, Def on p.80
Let $(P,\tau)$ be a non-separating Klein surface of type $(g,k,0)$.
Let $m$ be even.
The {\it topological type\/} of a real $m$-Arf function~$\arf$ on~$(P,\tau)$ is a tuple $(g,\de,k_0,k_1)$,
where $g$ is the genus of~$P$,
$\de$ is the $m$-Arf invariant of~$\arf$
and $k_j$ is the number of ovals of~$(P,\tau)$ with value of~$\arf$ equal to~$j\cdot m/2$.
\end{mydef}

Real $m$-Arf functions with even~$m$ on separating Klein surfaces have additional topological invariants:

\begin{mydef}
\label{def-similar}
% Nbook, Def on p.81
Let $(P,\tau)$ be a separating Klein surface of type $(g,k,1)$.
Let $P_1$ and~$P_2$ be the connected components of~$P\backslash P^{\tau}$.
Let $m$ be even.
Let $\arf$ be an $m$-Arf function on $(P,\tau)$.
We say that two ovals~$c_1$ and~$c_2$ are {\it similar\/} with respect to~$\arf$, $c_1\sim c_2$,
if $\arf(\ell\cup(\tau\ell)^{-1})$ is odd,
where $\ell$ is a simple path in~$P_1$ connecting $c_1$ and~$c_2$.
\end{mydef}

From the definition of $m$-Arf functions (see Definition~3{.}4 in~\cite{NP:2016}) it is clear
that if $\arf:\piorbO(P)\to\z/m\z$ is a real $m$-Arf function on~$(P,\tau)$ and $m$ is even,
then $(\arf~(\mod2)):\piorbO(P)\to\z/2\z$ is a real $2$-Arf function on~$(P,\tau)$.
Note that two ovals are similar with respect to the $m$-Arf function~$\arf$
if and only if they are similar with respect to the $2$-Arf function~$(\arf~(\mod2))$,
hence we obtain using~\cite{Nbook}, Theorem~3.3:

\begin{prop}
% Nbook, Thm 3.3 on p.81
\label{similarity}
Similarity of ovals is well-defined.
Similarity is an equivalence relation on the set of all ovals with at most two equivalence classes.
\end{prop}

\begin{mydef}
\label{def-type-sep-even}
% Nbook, Def on p.81
Let $(P,\tau)$ be a separating Klein surface of type $(g,k,1)$.
Let $P_1$ and~$P_2$ be the connected components of~$P\backslash P^{\tau}$.
Let $m$ be even.
Let us choose one similarity class of ovals.
The {\it topological type\/} of a real $m$-Arf function~$\arf$ on~$(P,\tau)$ is a tuple
$$(g,\tilde\de,k^0_0,k^0_1,k^1_0,k^1_1),$$
where $g$ is the genus of~$P$,
$\tilde\de$ is the $m$-Arf invariant of~$\arf|_{P_1}$,
$k^0_j$ is the number of ovals in the chosen similarity class with value of~$\arf$ equal to~$j\cdot m/2$
and $k^1_j=k_j-k_j^0$ is the number of ovals in the other similarity class with value of~$\arf$ equal to~$j\cdot m/2$.
(The invariants~$k_j^i$ are defined up to the swap $k_j^i\leftrightarrow k_j^{1-i}$.)
\end{mydef}

\begin{mydef}
% Nbook, Def on p.80
% Nbook, Def on p.81
Let $(P,\tau)$ be a Klein surface of type $(g,k,\ve)$.
Let $m$ be odd.
The {\it topological type\/} of a real $m$-Arf function~$\arf$ on~$(P,\tau)$ is a tuple $(g,k)$,
where $g$ is the genus of~$P$ and $k$ is the number of ovals of~$(P,\tau)$.
\end{mydef}

\begin{prop}
\label{topinv-nec}
% nonsep-even: Nbook, Thm 3.2 on p.80
% sep-even: Nbook, Thm 3.4 on p.81
% odd: Nbook, Thm 3.2 on p.80
If there exists a real $m$-Arf function of topological type $t$ on a Klein surface of type~$(g,k,\ve)$, $g\ge2$,
then $t$ satisfies the following conditions:
\begin{enumerate}[1)]
\item
Case $\ve=0$, $m\equiv0~(\mod2)$, $t=(g,\de,k_0,k_1)$: $k_1\cdot m/2\equiv1-g~(\mod m)$.
\item
Case $\ve=1$, $m\equiv0~(\mod2)$, $t=(g,\tilde\de,k^0_0,k^0_1,k^1_0,k^1_1)$:
Let $k_j=k_j^0+k_j^1$, $j=0,1$.
\begin{enumerate}[(a)]
\item
If $g>k+1$ and $m\equiv0~(\mod4)$ then $\tilde\de=0$.
\item
If $g>k+1$ and $k_0\ne0$ then $\tilde\de=0$.
\item
If $g=k+1$ and $m\equiv0~(\mod4)$ then $\tilde\de=1$.
\item
If $g=k+1$ and $k_0\ne0$ then $\tilde\de=1$.
\item
If $g=k+1$, $m\equiv2~(\mod4)$ and $k_0=0$ then $\tilde\de\in\{1,2\}$.
\item
$k_1\cdot m/2\equiv1-g~(\mod m)$.
\end{enumerate}
\item
Case $m\equiv1~(\mod2)$, $t=(g,k)$: $g\equiv1~(\mod m)$.
\end{enumerate}
\end{prop}

\begin{proof}
Let $(P,\tau)$ be a Klein surface of type~$(g,k,\ve)$, $g\ge2$.
Let $\arf$ be a real $m$-Arf function of topological type $t$ on $(P,\tau)$.
Let $c_1,\dots,c_k$ be the ovals of $(P,\tau)$.
\begin{enumerate}[1)]
\item
Case $\ve=0$, $m\equiv0~(\mod2)$, $t=(g,\de,k_0,k_1)$:
By definition of $k_j$,
the tuple $(\arf(c_1),\dots,\arf(c_k))$ is a permutation of zero repeated $k_0$ times and $m/2$ repeated $k_1$ times, hence $\sum\limits_{i=1}^k\,\arf(c_i)\equiv k_1\cdot m/2~(\mod m)$.
On the other hand Theorem~4{.}9(1) in~\cite{NP:2016} implies $\sum\limits_{i=1}^k\,\arf(c_i)\equiv1-g~(\mod m)$.
Hence $k_1\cdot m/2\equiv1-g~(\mod m)$.
\item
Case $\ve=1$, $m\equiv0~(\mod2)$, $t=(g,\tilde\de,k^0_0,k^0_1,k^1_0,k^1_1)$:
Let $P_1$ and~$P_2$ be the connected components of $P\backslash P^{\tau}$.
Each of these components is a surface of genus $\tilde g=(g+1-k)/2$ with $k$~holes.
If $\arf$ is a real $m$-Arf function of topological type $(g,\tilde\de,k^0_0,k^0_1,k^1_0,k^1_1)$ on $(P,\tau)$,
then $\arf|_{P_1}$ is an $m$-Arf function on a surface of genus $\tilde g$ with $k$~holes
with values on the holes equal to zero repeated $k_0$ times and $m/2$ repeated $k_1$ times. 
\begin{enumerate}[$\bullet$]
\item
Theorem~4{.}3(b) in~\cite{NP:2016} implies that if $\tilde g>1$ and $\arf(c_i)\equiv0~(\mod2)$ for some~$i$ then $\tilde\de=0$.
Note that $\tilde g>1$ if and only if $g>k+1$.
If $m\equiv0~(\mod4)$ then all $\arf(c_i)$ are even since both $0$ and $m/2$ are even, therefore $\tilde\de=0$.
If $k_0\ne0$ then $\arf(c_i)=0$ for some~$i$, hence $\arf(c_i)$ is even for some~$i$, therefore $\tilde\de=0$.
However, if $m\equiv2~(\mod4)$ and $k_0=0$ then all $\arf(c_i)=m/2$ are odd, hence no conclusion can be made about~$\tilde\de$. 
Thus we can rewrite the condition as follows:
If $g>k+1$ and ($m\equiv0~(\mod4)$ or $k_0\ne0$) then $\tilde\de=0$.
\item
Theorem~4{.}3(c) in~\cite{NP:2016} implies that in the case~$\tilde g=1$
the Arf invariant~$\tilde\de$ is a divisor of $\gcd(m,\arf(c_1)+1,\dots,\arf(c_k)+1)$.
Note that $\tilde g=1$ if and only if $g=k+1$.
If $k_0\ne0$ then $\arf(c_i)=0$ for some~$i$, hence $\tilde\de$ is a divisor of $\gcd(m,1,\dots)$, therefore $\tilde\de=1$.
If $k_0=0$ then $\arf(c_i)=m/2$ for all~$i$, hence $\tilde\de$ is a divisor of $\gcd\left(m,\frac{m}{2}+1\right)$.
For $m\equiv0~(\mod4)$ we have $\gcd\left(m,\frac{m}{2}+1\right)=1$, hence $\tilde\de=1$.
For $m\equiv2~(\mod4)$ we have $\gcd\left(m,\frac{m}{2}+1\right)=2$, hence $\tilde\de\in\{1,2\}$.
%If $\tilde\de$ is an odd divisor of~$m$ then $\tilde\de$ is also a divisor of $\frac{m}2$.
%But at the same time $\tilde\de$ is a divisor of~$\frac{m}2+1$ and therefore $\tilde\de=1$.
%If $\tilde\de$ is an even divisor of~$m$ then $\frac{\tilde\de}2$ is a divisor of $\frac{m}2$.
%But at the same time $\frac{\tilde\de}2$ is a divisor of~$\frac{m}2+1$ and therefore $\tilde\de=2$.
%Finally $\tilde\de=2$ is a divisor of $\frac{m}2+1$, hence $m\equiv2~(\mod4)$.
Therefore we can rewrite the condition as follows:
If $g=k+1$ and ($m\equiv0~(\mod4)$ or $k_0\ne0$) then $\tilde\de=1$.
If $g=k+1$, $m\equiv2~(\mod4)$ and $k_0=0$ then $\tilde\de\in\{1,2\}$.
\item
Theorem~4{.}3(d) in~\cite{NP:2016} implies that $\arf(c_1)+\cdots+\arf(c_k)\equiv(2-2\tilde g)-k~(\mod m)$.
Note that $\arf(c_1)+\cdots+\arf(c_k)=k_1\cdot m/2$ and $(2-2\tilde g)-k=1-g$.
Hence we can rewrite the condition as follows:
$k_1\cdot m/2\equiv1-g~(\mod m)$.
\end{enumerate}
\item
Case $m\equiv1~(\mod2)$, $t=(g,k)$:
Theorem~4{.}10(1) in~\cite{NP:2016} implies $g\equiv1~(\mod m)$.
\end{enumerate}
\end{proof}

%\begin{prop}
%Let $(P,\tau)$ be a Klein surface of type~$(g,k,0)$, $g\ge2$, and let $m$ be even.
%Let $\arf$ be an $m$-Arf function of type $(g,\de,k_0,k_1)$ on $(P,\tau)$.
%Let $P_1$ and $P_2$ be a decomposition of $(P,\tau)$ in halves.
%The Arf invariant $\tilde\de$ of $\arf|_{P_1}$ is given by
%\begin{align*}
% &\de=0\quad\text{if}\quad\tilde g\ge2,\\
%  &\de=1\quad\text{if}\quad\tilde g=1.
%\end{align*}
%\end{prop}

\begin{prop}
Let $(P,\tau)$ be a Klein surface of type~$(g,k,1)$, $g\ge2$, and let $m$ be even.
Let $\arf$ be an $m$-Arf function of type $(g,\tilde\de,k^0_0,k^0_1,k^1_0,k^1_1)$ on $(P,\tau)$.
Then the Arf invariant $\de\in\{0,1\}$ of~$\arf$ is given by
\begin{align*}
  &\de\equiv k_0^0\equiv k_0^1~(\mod2)\quad\text{if}\quad m\equiv2~(\mod4),\\
  &\de\equiv k_0^0+k_1^0\equiv k_0^1+k_1^1~(\mod2)\quad\text{if}\quad m\equiv0~(\mod4).
\end{align*}
\end{prop}

\begin{proof}
Consider an $m$-Arf function~$\arf$ of type $(g,\tilde\de,k^0_0,k^0_1,k^1_0,k^1_1)$ on $(P,\tau)$.
Let $c_1,\dots,c_k$ be the ovals of $(P,\tau)$.
We choose a symmetric generating set
$$
  \calB
  =
  (
   a_1,b_1,\dots,a_{\tilde g},b_{\tilde g},
   a_1',b_1',\dots,a_{\tilde g}',b_{\tilde g}',
   c_1,\dots,c_{k-1},d_1,\dots,d_{k-1}
  ).
$$
of~$\piorb(P)$.
Set $\ga_i=\arf(c_i)$ for $i=1,\dots,k$ and $\de_i=\arf(d_i)$ for $i=1,\dots,k-1$.
We can assume without loss of generality that the oval~$c_k$ is in the chosen similarity class (see Definition~\ref{def-similar}).
Let $\de_k=1$.
For $\al,\be\in\{0,1\}$ let $A_{\al}^{\be}$ be the subsets of $\{1,\dots,k\}$ given by
$$A_{\al}^{\be}=\{i\st\ga_i=\al\cdot m/2,\de_i\equiv1-\be~(\mod2)\}.$$
Then $k\in A_0^0\cup A_1^0$.
Note that $|A_{\al}^{\be}|=k_{\al}^{\be}$.
According to Theorem~4{.}9(4) in~\cite{NP:2016},
the Arf invariant~$\de$ of~$\arf$ is given by
$$\de\equiv\sum\limits_{i=1}^{k-1}(1-\ga_i)(1-\de_i)~(\mod2).$$
Weichold's classification of Klein surfaces implies $k\equiv g+1~(\mod2)$.
If~$m\equiv2~(\mod4)$, then
$$
  \sum\limits_{i=1}^{k-1}(1-\ga_i)(1-\de_i)
  \equiv|A_0^1\cap\{1,\dots,k-1\}|
  \equiv|A_0^1|
  \equiv k_0^1~(\mod2).
$$
In this case $m/2$ is odd, hence condition $k_1\cdot m/2\equiv1-g~(\mod m)$
can be reduced modulo~$2$ to $k_1\equiv1-g~(\mod2)$.
Using $k\equiv g+1~(\mod2)$ we obtain
$$k_0=k-k_1\equiv(g+1)-(1-g)\equiv0~(\mod2),$$
i.e.
$$k_0^1=k_0-k_0^0\equiv k_0^0~(\mod2).$$
If~$m\equiv0~(\mod4)$, then
$$
  \sum\limits_{i=1}^{k-1}(1-\ga_i)(1-\de_i)
  \equiv|(A_0^1\cup A_1^1)\cap\{1,\dots,k-1\}|
  \equiv|A_0^1\cup A_1^1|
  \equiv k_0^1+k_1^1~(\mod2).
$$
In this case $m/2$ is even, hence condition $k_1\cdot m/2\equiv1-g~(\mod m)$
can be reduced modulo~$2$ to $0\equiv1-g(\mod2)$,
hence $g$ is odd, so that $k\equiv g+1~(\mod2)$ is even.
Therefore
$$k_0^1+k_1^1=k-(k_0^0+k_1^0)\equiv k_0^0+k_1^0~(\mod2).$$
\end{proof}

\subsection{Canonical Symmetric Generating Sets}

\label{sec-canon}

For a Klein surface~$(P,\tau)$ we introduced in~\cite{NP:2016}
symmetric generating sets of $\piorb(P)$. 
These generating sets have certain symmetry with respect to the action of~$\tau$.
In this section we will construct for any real $m$-Arf function~$\arf$ a standard generating set of~$\piorb(P)$
on which $\arf$ assumes prescribed values determined by the topological invariants of~$\arf$.
We will call such a generating set canonical for~$\arf$.
For the convenience of the reader we will first recall the definition of a standard generating set.
The following fact is well known, see for example~\cite{Nbook,N1975,N1978moduli} and~\cite{BEGG}:

\begin{prop}
\label{prop-two-halves}
Let $(P,\tau)$ be a Klein surface of topological type~$(g,k,\ve)$.
Let $c_1,\dots,c_k$ be the ovals of~$(P,\tau)$.
In the case~$\ve=0$ we can choose for any $n$ with $k+1\le n\le g+1$ and $n\equiv g+1~(\mod2)$
twists $c_{k+1},\dots,c_n$ such that the complement of the curves $c_1,\dots,c_n$ in~$P$
consists of two components~$P_1$ and~$P_2$.
In the case~$\ve=1$ we can take $n=k$.
% as the complement of the curves $c_1,\dots,c_k$ in~$P$ consists of two components~$P_1$ and~$P_2$.
Each of the components~$P_1$ and~$P_2$ is a surface of genus $\tilde g=(g+1-n)/2$ with $n$~holes.
We will refer to $P_1$ and $P_2$ as a {\it decomposition of}~$(P,\tau)$ {\it in two halves}.
Note that such a decomposition is unique if $(P,\tau)$ is separating,
but is not unique if $(P,\tau)$ is non-separating since the twists $c_{k+1},\dots,c_n$ can be chosen in different ways.
\end{prop}

\begin{mydef}
\label{def-bridge}
Let $(P,\tau)$ be a Klein surface and $c_1,\dots,c_n$ invariant closed curves as in Proposition~\ref{prop-two-halves}
such that the complement of the curves $c_1,\dots,c_n$ in~$P$ consists of two components~$P_1$ and~$P_2$.
For two invariant closed curves~$c_i$ and~$c_j$, a {\it bridge} between~$c_i$ and~$c_j$ is a curve of the form
$$r_i\cup(\tau\ell)^{-1}\cup r_j\cup\ell,$$
where:
\begin{enumerate}[$\bullet$]
\item
$\ell$ is a simple path in~$P_1$ starting on~$c_j$ and ending on~$c_i$.
\item
$r_i$ is the path along~$c_i$ from the end point of~$\ell$ to the end point of~$\tau\ell$.
(If $c_i$ is an oval then the path $r_i$ consists of one point.)
\item
$r_j$ is the path along~$c_j$ from the starting point of~$\tau\ell$ to the starting point of~$\ell$.
(If $c_j$ is an oval then the path $r_j$ consists of one point.)
\end{enumerate}
Figure~\ref{fig-bridges} shows the shapes of the bridges for different types of invariant curves.
The bridges are shown in bold.
The bold arrows on the bold lines show the direction of the bridges,
while the thinner arrows near the lines show the directions of the paths $c_i$, $c_j$, $r_i$, $r_j$, $\ell$ and $\tau\ell$. 
\end{mydef}

% FIGURE
\begin{figure}[H]
\begin{center}
\leavevmode
  \setcoordinatesystem units <1cm,1cm> point at 0 0
  \setplotarea x from -8 to 8, y from -4 to 3
  \plot -6 -4 6 -4 6 3 -6 3 -6 -4 /
  \plot -2.5 -4 -2.5 3 /
  \plot 1.5 -4 1.5 3 /
%%%%%%%%%%%%%%%%%%%%%%%%%%%
  \ellipticalarc axes ratio 1:2 360 degrees from -4.25 1 center at -4.25 1.5
%  \arrow <10pt> [0.2,0.5] from -4 1.3 to -4 1.7
  \arrow <10pt> [0.2,0.5] from -4.7 1.7 to -4.7 1.3
  \ellipticalarc axes ratio 1:2 360 degrees from -4.25 -1 center at -4.25 -1.5
%  \arrow <10pt> [0.2,0.5] from -4 -1.7 to -4 -1.3
  \arrow <10pt> [0.2,0.5] from -4.7 -1.3 to -4.7 -1.7
  \arrow <10pt> [0.2,0.5] from -5.15 -0.2 to -5.15 0.2
  \arrow <10pt> [0.2,0.5] from -3.35 -0.2 to -3.35 0.2
  \put {$c_i$} [r] <-5pt,0pt> at -4.7 1.5
  \put {$c_j$} [r] <-5pt,0pt> at -4.7 -1.5
  \put {$r_i$} [l] <10pt,0pt> at -4.25 1
  \put {$r_j$} [l] <10pt,0pt> at -4.25 -1
  \put {$\ell$} [r] <-5pt,0pt> at -5.15 0
  \put {$\tau\ell$} [l] <5pt,0pt> at -3.35 0
  \put {$c_i,c_j$~\text{ovals}} [c] at -4.25 -3
%  \multiput {$\bullet$} at -4.25 1 /
%%%%%%%%%%%%%%%%%%%%%%%%%%%
  \ellipticalarc axes ratio 1:2 360 degrees from -1 1 center at -1 1.5
%  \arrow <10pt> [0.2,0.5] from -0.75 1.3 to -0.75 1.7
  \arrow <10pt> [0.2,0.5] from -1.45 1.7 to -1.45 1.3
  \ellipticalarc axes ratio 1:2 360 degrees from -1 -1 center at -1 -1.5
%  \arrow <10pt> [0.2,0.5] from -0.75 -1.7 to -0.75 -1.3
  \arrow <10pt> [0.2,0.5] from -1.45 -1.3 to -1.45 -1.7
  \arrow <10pt> [0.2,0.5] from -1.9 -0.2 to -1.9 0.2
  \arrow <10pt> [0.2,0.5] from 0.8 -0.7 to 0.8 -0.3
  \arrow <10pt> [0.2,0.5] from -0.55 -1.7 to -0.55 -1.3
  \put {$c_i$} [r] <-5pt,0pt> at -1.45 1.5
  \put {$c_j$} [r] <-5pt,0pt> at -1.45 -1.5
  \put {$r_i$} [l] <10pt,4pt> at -1 1
  \put {$r_j$} [l] <5pt,0pt> at -0.55 -1.5
  \put {$\ell$} [r] <-5pt,0pt> at -1.9 0
  \put {$\tau\ell$} [l] <5pt,0pt> at 0.8 -0.5
  \put {$c_i$~\text{oval}, $c_j$~\text{twist}} [c] at -0.5 -3
%  \multiput {$\bullet$} at -1 1 /
%%%%%%%%%%%%%%%%%%%%%%%%%%%
  \ellipticalarc axes ratio 1:2 360 degrees from 3 1 center at 3 1.5
  \ellipticalarc axes ratio 1:2 360 degrees from 3 -1 center at 3 -1.5
%  \arrow <10pt> [0.2,0.5] from 3.25 1.3 to 3.25 1.7
  \arrow <10pt> [0.2,0.5] from 2.55 1.7 to 2.55 1.3
%  \arrow <10pt> [0.2,0.5] from 3.25 -1.7 to 3.25 -1.3
  \arrow <10pt> [0.2,0.5] from 2.55 -1.3 to 2.55 -1.7
  \arrow <10pt> [0.2,0.5] from 2.1 -0.2 to 2.1 0.2
  \arrow <10pt> [0.2,0.5] from 5.4 -0.2 to 5.4 0.2
  \arrow <10pt> [0.2,0.5] from 3.45 -1.7 to 3.45 -1.3
  \arrow <10pt> [0.2,0.5] from 3.45 1.3 to 3.45 1.7
  \put {$c_i$} [r] <-5pt,0pt> at 2.45 1.5
  \put {$c_j$} [r] <-5pt,0pt> at 2.45 -1.5
  \put {$\ell$} [r] <-5pt,0pt> at 2.1 0
  \put {$\tau\ell$} [l] <5pt,0pt> at 5.4 0
  \put {$r_i$} [l] <5pt,0pt> at 3.45 1.5
  \put {$r_j$} [l] <5pt,0pt> at 3.45 -1.5
  \put {$c_i,c_j$~\text{twists}} [c] at 3.75 -3
%  \multiput {$\bullet$} at 3 1 /
%%%%%%%%%%%%%%%%%%%%%%%%%%%
  \setplotsymbol({$\bullet$})
%%%%%%%%%%%%%%%%%%%%%%%%%%%
  \ellipticalarc axes ratio 1:2 360 degrees from -4.25 1 center at -4.25 0
  \arrow <10pt> [0.3,0.8] from -4.75 -0.3 to -4.75 0.3
  \arrow <10pt> [0.3,0.8] from -3.75 0.3 to -3.75 -0.3
%%%%%%%%%%%%%%%%%%%%%%%%%%%
  \ellipticalarc axes ratio 1:2 180 degrees from -1 1 center at -1 0
  \ellipticalarc axes ratio 1:1 180 degrees from -1 -2 center at -1 -0.5
  \ellipticalarc axes ratio 1:2 -180 degrees from -1 -1 center at -1 -1.5
  \arrow <10pt> [0.3,0.8] from -1.5 -0.2 to -1.5 0.2
  \arrow <10pt> [0.3,0.8] from 0.5 -0.3 to 0.5 -0.7
%%%%%%%%%%%%%%%%%%%%%%%%%%%
  \ellipticalarc axes ratio 1:2 180 degrees from 3 1 center at 3 0
  \ellipticalarc axes ratio 1:1 180 degrees from 3 -2 center at 3 0
  \ellipticalarc axes ratio 1:2 180 degrees from 3 1 center at 3 1.5
  \ellipticalarc axes ratio 1:2 -180 degrees from 3 -1 center at 3 -1.5
  \arrow <10pt> [0.3,0.8] from 2.5 -0.2 to 2.5 0.2
  \arrow <10pt> [0.3,0.8] from 5 0.2 to 5 -0.2
%  \arrow <10pt> [0.3,0.8] from 3.25 1.3 to 3.25 1.7
%  \arrow <10pt> [0.3,0.8] from 3.25 -1.7 to 3.25 -1.3
\end{center}
\caption{Bridges}
\label{fig-bridges}
\end{figure}

\begin{mydef}
\label{def-symm-gen}
Let $(P,\tau)$ be a Klein surface of topological type~$(g,k,\ve)$.
A {\it symmetric generating set} of~$\piorb(P)$ is a generating set of the form
$$
  (
   a_1,b_1,\dots,a_{\tilde g},b_{\tilde g},
   a_1',b_1',\dots,a_{\tilde g}',b_{\tilde g}',
   c_1,\dots,c_{n-1},d_1,\dots,d_{n-1}
  ),
$$
where
\begin{enumerate}[$\bullet$]
\item
$n=k$ if $\ve=1$.
\item
$k+1\le n\le g+1$ and $n\equiv g+1~(\mod2)$ if~$\ve=0$.
\item
$c_1,\dots,c_k$ are the ovals of~$(P,\tau)$.
\item
$c_{k+1},\dots,c_{n-1}$ are twists (in the case~$\ve=0$).
\item
There exists an invariant closed curve~$c_n$ such that the complement of the curves $c_1,\dots,c_n$ in~$P$
consists of two components~$P_1$ and~$P_2$.
The invariant curve~$c_n$ is an oval if~$\ve=1$ and a twist if~$\ve=0$.
\item
$(a_1,b_1,\dots,a_{\tilde g},b_{\tilde g},c_1,\dots,c_n)$ is a generating set of $\piorb(P_1)$.
\item
$a_i'=(\tau a_i)^{-1}$ and $b_i'=(\tau b_i)^{-1}$ for $i=1,\dots,\tilde g$.
\item
$d_1,\dots,d_{n-1}$ are closed curves which only intersect at the base point,
such that $d_i$ is homotopic to a bridge between $c_i$ and $c_n$,
\end{enumerate}
Note that $\tau c_i=c_i$ and $\tau d_i=c_i^{|c_i|}d_i^{-1}c_n^{|c_n|}$,
where $|c_j|=0$ if $c_j$ is an oval and $|c_j|=1$ if $c_j$ is a twist.
\end{mydef}

\begin{mydef}
\label{def-canon}
Let $(P,\tau)$ be a Klein surface of type $(g,k,\ve)$, $g\ge2$,
and $\arf$ a real $m$-Arf function~$\arf$ of topological type $t$ on~$(P,\tau)$.
Let 
$$
  \calB
  =(a_1,b_1,\dots,a_{\tilde g},b_{\tilde g},a_1',b_1',\dots,a_{\tilde g}',b_{\tilde g}',
      c_1,\dots,c_{n-1},d_1 \dots,d_{n-1})
$$
be a symmetric generating set of~$\piorb(P)$ and $$\al_i=\arf(a_i),~\be_i=\arf(b_i),~\al_i'=\arf(a_i'),~\be_i'=\arf(b_i'),~\ga_i=\arf(c_i),~\de_i=\arf(d_i).$$
We say that~$\calB$ is {\it canonical} for the $m$-Arf function~$\arf$ if
\begin{enumerate}[$\bullet$]
\item
Case $\ve=0$, $m\equiv0~(\mod2)$, $t=(g,\de,k_0,k_1)$:
\begin{align*}
  &(\al_1,\be_1,\dots,\al_{\tilde g},\be_{\tilde g})
  =(\al_1',\be_1',\dots,\al_{\tilde g}',\be_{\tilde g}')
  =(0,1,1,\dots,1)~\text{if}~\tilde g\ge2,\\
  &(\al_1,\be_1)=(\al_1',\be_1')=(1,0)~\text{if}~\tilde g=1,\\
  &\ga_1=\cdots=\ga_{k_0}=0,\quad\ga_{k_0+1}=\cdots=\ga_k=m/2,\quad\ga_{k+1}=\cdots=\ga_{n-1}=0,\\
  &\de_1=\cdots=\de_{n-1}=1-\de.
\end{align*}
\item
Case $\ve=1$, $m\equiv0~(\mod2)$, $t=(g,\tilde\de,k^0_0,k^0_1,k^1_0,k^1_1)$:
\begin{align*}
  &(\al_1,\be_1,\dots,\al_{\tilde g},\be_{\tilde g})
  =(\al_1',\be_1',\dots,\al_{\tilde g}',\be_{\tilde g}')
  =(0,1-\tilde\de,1,\dots,1)~\text{if}~\tilde g\ge2;\\
  &(\al_1,\be_1)=(\al_1',\be_1')=(\tilde\de,0)~\text{if}~\tilde g=1;\\
  &\ga_1=\cdots=\ga_{k_0}=0,\quad\ga_{k_0+1}=\cdots=\ga_{k-1}=m/2;\\
  &\text{The oval}~c_k~\text{is in the chosen similarity class};\\
  &\de_1=\cdots=\de_{k_0^1}=0,\quad\de_{k_0^1+1}=\cdots=\de_{k_0}=1,\\
  &\de_{k_0+1}=\cdots=\de_{k_0+k_1^1}=0,\quad\de_{k_0+k_1^1+1}=\cdots=\de_{k-1}=1~\text{if}~k_1\ge1;\\
  &\de_1=\cdots=\de_{k_0^1}=0,\quad\de_{k_0^1+1}=\cdots=\de_{k-1}=1~\text{if}~k_1=0.
\end{align*}
%If $k_1\ne0$ and the oval~$c_k$ is not in the chosen similarity class then 
%\begin{align*}
% &\de_1=\cdots=\de_{k_0^0}=1,\quad\de_{k_0^0+1}=\cdots=\de_{k_0^0+k_0^1}=0,\\
%  &\de_{k_0^0+k_0^1+1}=\cdots=\de_{k_0^0+k_0^1+k_1^0}=1,\quad\de_{k_0^0+k_0^1+k_1^0+1}=\cdots=\de_{k-1}=0.
%\end{align*}
%If $k_1=0$ and the oval~$c_k$ is not in the chosen similarity class then 
%$$\de_1=\cdots=\de_{k_0^0}=1,\quad\de_{k_0^0+1}=\cdots=\de_{k-1}=0.$$
\item
Case $m\equiv1~(\mod2)$, $t=(g,k)$:
\begin{align*}
  &(\al_1,\be_1,\dots,\al_{\tilde g},\be_{\tilde g})
  =(\al_1',\be_1',\dots,\al_{\tilde g}',\be_{\tilde g}')
  =(0,1,1,\dots,1)~\text{if}~\tilde g\ge2,\\
  &(\al_1,\be_1)=(\al_1',\be_1')=(1,0)~\text{if}~\tilde g=1,\\
  &\ga_1=\cdots=\ga_{n-1}=0,\\
  &\de_1=\cdots=\de_{n-1}=0.
\end{align*}
\end{enumerate}
\end{mydef}

\begin{lem}
\label{lem-modify-bridges}
Let $(P,\tau)$ be a Klein surface of type $(g,k,\ve)$, $g\ge2$.
Let the geometric genus of $(P,\tau)$ be positive, i.e.\ $k\le g-1$ if $\ve=1$ and $k\le g-2$ if $\ve=0$. 
In the case $\ve=1$ let $n=k$.
In the case $\ve=0$ we choose $n\in\{k+1,\dots,g-1\}$ such that $n\equiv g-1~(\mod2)$.
(The assumption that the geometric genus is positive implies $k+1\le g-1$, hence $\{k+1,\dots,g-1\}\ne\emptyset$.)
Let $c_1,\dots,c_n$ be invariant closed curves as in Proposition~\ref{prop-two-halves},
then bridges $d_1,\dots,d_{n-1}$ as in Definition~\ref{def-symm-gen} can be chosen in such a way that
\begin{enumerate}[(i)]
\item
If $m$ is odd, then $\arf(d_i)=0$ for~$i=1,\dots,n-1$.
\item
If $m$ is even and $(P,\tau)$ is separating, then $\arf(d_i)\in\{0,1\}$ for $i=1,\dots,n-1$.
\item
If $m$ is even and $(P,\tau)$ is non-separating, then $\arf(d_1)=\cdots=\arf(d_{n-1})\in\{0,1\}$.
\end{enumerate}
\end{lem}

\begin{proof}
Let $P_1$ and~$P_2$ be the connected components of the complement of the closed curves~$c_1,\dots,c_n$ in~$P$.
Each of these components is a surface of genus $\tilde g=(g+1-n)/2$ with $n$~holes.
The assumption $n\le g-1$ implies $\tilde g\ge1$.
\begin{enumerate}[$\bullet$]
\item
Consider the real $2$-Arf function $(\arf~(\mod2)):\piorbO(P)\to\z/2\z$.
If $m$ is even and $(P,\tau)$ is non-separating, then, according to Lemma~11.2 in~\cite{Nbook},
we can choose the bridges $d_1,\dots,d_{n-1}$ in such a way that
$$(\arf~(\mod2))(d_1)=\cdots=(\arf~(\mod2))(d_{n-1}).$$
This means for the original $m$-Arf function~$\arf$ that
$$\arf(d_1)\equiv\cdots\equiv\arf(d_{n-1})~(\mod2).$$
\item
Let $Q_1$ be the compact surface of genus~$\tilde g$ with one hole
obtained from~$P_1$ after removing all bridges $d_1,\dots,d_{n-1}$.
Let $\tilde\de$ be the Arf invariant of~$\arf|_{Q_1}$.
In the case $\tilde g\ge2$, Lemma~5.1 in~\cite{NP:2009} implies
that we can choose a standard generating set $(a_1,b_1,\dots,a_{\tilde g},b_{\tilde g},\tilde c)$ of~$\piorb(Q_1)$
in such a way that $\arf(a_1)=0$.
In the case $\tilde g=1$, Lemma~5.2 in~\cite{NP:2009} implies 
that we can choose a standard generating set $(a_1,b_1,\tilde c)$ of~$\piorb(Q_1)$ 
in such a way that $\arf(b_1)=0$.
Thus for $\tilde g\ge1$ there always exists a non-trivial closed curve~$a$ in~$P_1$ with $\arf(a)=0$,
which does not intersect any of the bridges $d_1,\dots,d_{n-1}$.
If we replace $d_i$ by $(\tau a)^{-1} d_i a$, then
$$\arf((\tau a)^{-1} d_i a)=\arf((\tau a)^{-1})+\arf(d_i)+\arf(a)-2.$$
Taking into account the fact that $\arf(a)=0$ we obtain
$$\arf((\tau a)^{-1} d_i a)=\arf(d_i)-2.$$
Repeating this operation we can obtain $\arf(d_i)=0$ for odd~$m$ and $\arf(d_i)\in\{0,1\}$ for even~$m$.
\item
Note that the property $\arf(d_1)\equiv\cdots\equiv\arf(d_{n-1})~(\mod2)$ (if $m$ is even and $(P,\tau)$ is non-separating)
is preserved during this process, hence $\arf(d_1)=\cdots=\arf(d_{n-1})$ at the end of the process.
\end{enumerate}
\end{proof}

\begin{prop}
\label{prop-canon1}
Let $(P,\tau)$ be a Klein surface of positive geometric genus. % LIMIT
% i.e.\ $k\le g-2$ if $\ve=0$ and $k\le g-1$ if $\ve=1$, where $(g,k,\ve)$ is the topological type of $(P,\tau)$. % LIMIT 
For any real $m$-Arf function on $(P,\tau)$ there exists a canonical symmetric generating set of $\piorb(P)$.
\end{prop}

\begin{proof}
Let $(g,k,\ve)$ be the topological type of the Klein surface $(P,\tau)$.
Let $\arf$ be a real $m$-Arf function on~$(P,\tau)$.
Let $c_1,\dots,c_n$ be invariant closed curves as in Proposition~\ref{prop-two-halves}.
% <LIMIT--
%\begin{enumerate}[1)]
%\item
%We first consider the case that the conditions of Lemma~\ref{lem-modify-bridges} are satisfied,
%i.e.~$k\le g-2$ for~$\ve=0$ and $k\le g-1$ for $\ve=1$, then 
% --LIMIT>
\begin{enumerate}[$\bullet$]
\item
If $m\equiv0~(\mod2)$ then $\arf(c_{k+1})=\dots=\arf(c_n)=0$.
\item
If $m\equiv0~(\mod2)$ then $\arf(c_1),\dots,\arf(c_k)\in\{0,m/2\}$.
We can reorder the ovals $c_1,\dots,c_k$ in such a way that 
$$\arf(c_1)=\cdots=\arf(c_{k_0})=0,\quad\arf(c_{k_0+1})=\cdots=\arf(c_k)=m/2,$$
where $k_0$ is the numbers of ovals of~$(P,\tau)$ with the value of~$\arf$ equal to~$0$.
\item
If $m\equiv1~(\mod2)$ then $\arf(c_1)=\dots=\arf(c_n)=0$.
\item
We can choose bridges $d_1,\dots,d_{n-1}$ with values $\arf(d_i)$ as described in Lemma~\ref{lem-modify-bridges}
since the assumptions of the Lemma are satisfied.
\item
If $\ve=1$ and $m\equiv0~(\mod2)$,
we can change the order of $c_1,\dots,c_{k_0}$ and $c_{k_0+1},\dots,c_k$ to obtain the required values $\de_1,\dots,\de_{k-1}$.
\item
If $\ve=0$ and $m\equiv0~(\mod2)$, there exists $\xi\in\{0,1\}$ such that
$$\arf(d_1)=\cdots=\arf(d_{n-1})=\xi.$$
According to Theorem~4{.}9(4) in~\cite{NP:2016} the Arf invariant of $\arf$ is 
$$\de\equiv\sum\limits_{i=1}^{n-1}(1-\arf(c_i))(1-\arf(d_i))~(\mod2).$$
Using $\arf(d_i)=\xi$ we obtain
\begin{align*}
  \de
  &\equiv\sum\limits_{i=1}^{n-1}(1-\arf(c_i))(1-\arf(d_i))\\
%  &\equiv\sum\limits_{i=1}^{n-1}(1-\arf(c_i))(1-\xi)
  &\equiv(1-\xi)\cdot\sum\limits_{i=1}^{n-1}(1-\arf(c_i))\\
  &\equiv(1-\xi)\cdot\left((n-1)-\sum\limits_{i=1}^{n-1}\arf(c_i)\right)\\
  &\equiv(1-\xi)\cdot\left((n-1)-k_1\cdot\frac{m}2\right)~(\mod2).
%% since $\arf(c_i)=m/2$ for $i\in A$, $|A|=k_1$, and $\arf(c_i)=0$ for $i\not\in A$
\end{align*}
Recall that $k_1\cdot m/2\equiv1-g~(\mod m)$ by Proposition~\ref{topinv-nec} and $n\equiv g-1~(\mod2)$, hence
$$
  (n-1)-k_1\cdot\frac{m}2
  \equiv(g-2)-(1-g)
  \equiv2g-3
  \equiv1~(\mod2)
$$
and
$$\de\equiv(1-\xi)\cdot\left((n-1)-k_1\cdot\frac{m}2\right)\equiv1-\xi~(\mod2).$$
Therefore
$$\arf(d_1)=\cdots=\arf(d_{n-1})=\xi=1-\de.$$
\item
For $\tilde g\ge2$, Lemma~5.1 in~\cite{NP:2009} implies that we can choose
a standard generating set $(a_1,b_1,\dots,a_{\tilde g},b_{\tilde g},c_1,\dots,c_n)$ of $\piorb(P_1)$ in such a way that
$$(\arf(a_1),\arf(b_1),\dots,\arf(a_{\tilde g}),\arf(b_{\tilde g}))=(0,1-\tilde\de,1,\dots,1),$$
where $\tilde\de$ is the Arf invariant of~$\arf|_{P_1}$.
Moreover, if $m$ is odd then $\tilde\de=0$.
If $m$ is even and $\ve=0$ then there are closed curves around holes in $P_1$
such that the values of $\arf$ on these closed curves are even, namely $\arf(c_{k+1})=\cdots=\arf(c_n)=0$,
hence $\tilde\de=0$.
\item
If $\tilde g=1$, Lemma~5.2 in~\cite{NP:2009} implies that we can choose a standard generating set
$(a_1,b_1,c_1,\dots,c_n)$ of~$\piorb(P_1)$ in such a way that 
$$(\arf(a_1),\arf(b_1))=(\tilde\de,0),$$
where $\tilde\de=\gcd(m,\arf(a_1),\arf(b_1),\arf(c_1)+1,\dots,\arf(c_n)+1)$ is the Arf invariant of~$\arf|_{P_1}$.
If $m$ is odd then $\arf(c_1)=\cdots=\arf(c_n)=0$, hence $\tilde\de=1$.
If $\ve=0$ then $\arf(c_{k+1})=\cdots=\arf(c_n)=0$, hence $\tilde\de=1$.
\end{enumerate}
% <LIMIT-- 
%\item Case $\ve=0$, $k\in\{g-1,g\}$: XXX
%\item Case $\ve=1$, $k=g+1$: XXX
%\end{enumerate}
% --LIMIT>
\end{proof}

\begin{prop}
\label{prop-canon2}
For any Klein surface $(P,\tau)$ and any symmetric generating set~$\calB$ of~$\piorb(P)$ and any tuple~$t$
that satisfies the conditions of Proposition~\ref{topinv-nec}
there exists a real $m$-Arf function of topological type~$t$ on~$(P,\tau)$ for which $\calB$ is canonical.
\end{prop}

\begin{proof}
Let $\calV=(\al_i,\be_i,\al_i',\be_i',\ga_i,\de_i)$ satisfy the conditions in Definition~\ref{def-canon}.
\begin{enumerate}[$\bullet$]
\item
Case $\ve=0$, $m\equiv0~(\mod2)$, $t=(g,\de,k_0,k_1)$:
We have $\ga_1=\cdots=\ga_{k_0}=0$, $\ga_{k_0+1}=\cdots=\ga_{k_0+k_1}=m/2$,
hence $\sum\limits_{i=1}^k\,\ga_i=k_1\cdot m/2$.
The tuple~$t$ satisfies the conditions of Proposition~\ref{topinv-nec}, hence $k_1\cdot m/2\equiv1-g~(\mod m)$.
Therefore $\sum\limits_{i=1}^k\,\ga_i\equiv1-g~(\mod m)$.
Other conditions of Theorem~4{.}9(2) in~\cite{NP:2016} are clearly satisfied.
Hence there exists a real $m$-Arf function~$\arf$ on~$P$ with the values~$\calV$ on~$\calB$.
Let $\de'$ be the Arf invariant of~$\arf$, then
\begin{align*}
  \de'
  &\equiv\sum\limits_{i=1}^{n-1}(1-\ga_i)(1-\de_i)
  \equiv\sum\limits_{i=1}^{n-1}(1-\ga_i)(1-(1-\de))\\
  &\equiv\de\cdot\sum\limits_{i=1}^{n-1}(1-\ga_i)
  \equiv\de\cdot\left((n-1)-\sum\limits_{i=1}^{n-1}\ga_i\right)\\
  &\equiv\de\cdot\left((n-1)-k_1\cdot\frac{m}{2}\right)~(\mod2).
\end{align*}
Recall that $k_1\cdot m/2\equiv1-g~(\mod m)$ and $n\equiv g-1~(\mod2)$, hence
$$(n-1)-k_1\cdot\frac{m}{2}\equiv(g-2)-(1-g)\equiv2g-3\equiv1~(\mod2)$$
and
$$\de'\equiv\de\cdot\left((n-1)-k_1\cdot\frac{m}{2}\right)\equiv\de~(\mod2).$$
Hence $\arf$ is a real $m$-Arf function on~$P$ of type~$t$ and $\calB$ is canonical for $\arf$.
\item
Case $\ve=1$, $m\equiv0~(\mod2)$, $t=(g,\tilde\de,k^0_0,k^0_1,k^1_0,k^1_1)$:
The tuple~$t$ satisfies the conditions of Proposition~\ref{topinv-nec}, hence
$$1-g\equiv k_1\cdot\frac{m}{2}~(\mod m)$$
and therefore 
$$1-g\equiv0~(\mod\frac{m}{2}).$$
Other conditions of Theorem~4{.}9(2) in~\cite{NP:2016} are clearly satisfied.
Hence there exists a real $m$-Arf function~$\arf$ on~$P$ with the values~$\calV$ on~$\calB$.
Let $\tilde\de'$ be the Arf invariant of~$\arf|_{P_1}$.
The $m$-Arf function~$\arf$ is real, hence according to Proposition~\ref{topinv-nec}, we have
\begin{enumerate}[$\bullet$]
\item
If $g>k+1$ and $m\equiv0~(\mod4)$ then $\tilde\de'=0$.
\item
If $g>k+1$ and $k_0\ne0$ then $\tilde\de'=0$.
\item
If $g=k+1$ and $m\equiv0~(\mod4)$ then $\tilde\de'=1$.
\item
If $g=k+1$ and $k_0\ne0$ then $\tilde\de'=1$.
\item
If $g=k+1$, $m\equiv2~(\mod4)$ and $k_0=0$ then $\tilde\de'\in\{1,2\}$.
\end{enumerate}
On the other hand $t=(g,\tilde\de,k^0_0,k^0_1,k^1_0,k^1_1)$ satisfies the conditions of Proposition~\ref{topinv-nec},
hence
\begin{enumerate}[$\bullet$]
\item
If $g>k+1$ and $m\equiv0~(\mod4)$ then $\tilde\de=0$.
\item
If $g>k+1$ and $k_0\ne0$ then $\tilde\de=0$.
\item
If $g=k+1$ and $m\equiv0~(\mod4)$ then $\tilde\de=1$.
\item
If $g=k+1$ and $k_0\ne0$ then $\tilde\de=1$.
\item
If $g=k+1$, $m\equiv2~(\mod4)$ and $k_0=0$ then $\tilde\de\in\{1,2\}$.
\end{enumerate}
Hence if $m\equiv0~(\mod4)$ or $k_0\ne0$ we have $\tilde\de'=\tilde\de$.
It remains to consider the case $m\equiv2~(\mod4)$, $k_0=0$.
In the case $g>k+1$, $m\equiv2~(\mod4)$, $k_0=0$, we have $\tilde g\ge2$
and the values of the $m$-Arf function~$\arf|_{P_1}$ on the boundary curves $\arf(c_i)$ are all equal to~$m/2$ and hence odd.
Then, according to Theorem~4{.}4(c) in~\cite{NP:2016}, the Arf invariant~$\tilde\de'$ is given by
$$\tilde\de'\equiv\sum\limits_{i=1}^{\tilde g}(1-\al_i)(1-\be_i)~(\mod2).$$
We have $(\al_1,\be_1,\dots,\al_{\tilde g},\be_{\tilde g})=(0,1-\tilde\de,1,\dots,1)$, hence
$$
  \tilde\de'
  \equiv\sum\limits_{i=1}^{\tilde g}(1-\al_i)(1-\be_i)
  \equiv1\cdot\tilde\de+0+\cdots+0
  \equiv\tilde\de~(\mod2)
$$
and therefore $\tilde\de'=\tilde\de$.
In the case $g=k+1$, $m\equiv2~(\mod4)$, $k_0=0$, we have $\tilde g=1$
and the values of the $m$-Arf function~$\arf|_{P_1}$ on the boundary curves $\arf(c_i)$ are all equal to~$m/2$.
Then, according to Theorem~4{.}4(d) in~\cite{NP:2016}, the Arf invariant~$\tilde\de'\in\{1,2\}$ is given by
$$\tilde\de'=\gcd\left(m,\al_1,\be_1,\frac{m}{2}+1\right).$$
We have $(\al_1,\be_1)=(\tilde\de,0)$, hence $\gcd(\al_1,\be_1)=\tilde\de\in\{1,2\}$.
For $m\equiv2~(\mod4)$ we have $\gcd\left(m,\frac{m}{2}+1\right)=2$.
Therefore
$$\tilde\de'=\gcd\left(m,\al_1,\be_1,\frac{m}{2}+1\right)=\gcd(\tilde\de,2)=\tilde\de.$$
Hence $\arf$ is a real $m$-Arf function on~$P$ of type~$t$ and $\calB$ is canonical for $\arf$.
\item
Case $m\equiv1~(\mod2)$, $t=(g,k)$:
The tuple~$t$ satisfies the conditions of Proposition~\ref{topinv-nec}, hence $g\equiv1~(\mod m)$.
Other conditions of Theorem~4{.}10(2) in~\cite{NP:2016} are clearly satisfied.
Hence there exists a real $m$-Arf function~$\arf$ on~$P$ with the values~$\calV$ on~$\calB$.
The topological type of $\arf$ is $t$ and $\calB$ is canonical for $\arf$.
\end{enumerate}
\end{proof}

\begin{prop}
\label{topinv-suff}
The conditions in Proposition~\ref{topinv-nec} are necessary and sufficient
for a tuple to be a topological type of a real $m$-Arf function.
\end{prop}

\begin{proof}
Proposition~\ref{topinv-nec} shows that the conditions are necessary.
Proposition~\ref{prop-canon2} shows that the conditions are sufficient
as we constructed an $m$-Arf function of type~$t$ for any tuple~$t$ that satisfies the conditions.
\end{proof}

\begin{mydef}
\label{def-topeq-Arf}
% Nbook, Def on p.112
Two $m$-Arf functions~$\arf_1$ and~$\arf_2$ on a Klein surface $(P,\tau)$ are {\it topologically equivalent\/}
if there exists a homeomorphism~$\varphi:P\to P$ such that $\varphi\circ\tau=\tau\circ\varphi$ and
$\arf_1=\arf_2\circ\varphi_*$ for the induced automorphism~$\varphi_*$ of~$\piorb(P)$.
\end{mydef}

\begin{prop}
\label{topeq}
% Nbook, Thm 11.2, p.112
Let $(P,\tau)$ be a Klein surface of positive geometric genus. % LIMIT
% i.e.\ $k\le g-2$ if $\ve=0$ and $k\le g-1$ if $\ve=1$, where $(g,k,\ve)$ is the topological type of $(P,\tau)$. % LIMIT 
Two $m$-Arf functions on $(P,\tau)$ are topologically equivalent if and only if they have the same topological type.
\end{prop}

\begin{proof}
Let $(g,k,\ve)$ be the topological type of the Klein surface $(P,\tau)$.
Proposition~\ref{prop-canon1} shows that for any real $m$-Arf function~$\arf$ of topological type~$t$
we can choose a symmetric generating set~$\calB$ (the canonical generating set for~$\arf$)
with the values of~$\arf$ on~$\calB$ determined completely by~$t$.
Hence any two real $m$-Arf functions of topological type~$t$ are topologically equivalent.
\end{proof}

\section{Moduli Spaces}

\label{sec-moduli}

%\subsection{Moduli Spaces of Klein Surfaces}

We will use the results on the moduli spaces of real Fuchsian groups and of Klein surfaces
described in~\cite{N1975, N1978moduli}:
We consider hyperbolic Klein surfaces, i.e. we assume that the genus is~$g\ge2$.
Let $\calM_{g,k,\ve}$ be the moduli space of Klein surfaces of topological type~$(g,k,\ve)$.
Let $\Ga_{g,n}$ be the group generated by the elements
$$v=\{a_1,b_1,\dots,a_g,b_g,c_1,\dots,c_n\}$$
with a single defining relation
$$\prod\limits_{i=1}^g\,[a_i,b_i]\prod\limits_{i=1}^n\,c_i=1.$$
Let $\Aut_+(\hyp)$ be the group of all orientation-preserving isometries of~$\hyp$.
The {\it Fricke space}~$\tT_{g,n}$ is the set of all monomorphisms $\psi:\Ga_{g,n}\to\Aut_+(\hyp)$
such that
$$\{\psi(a_1),\psi(b_1),\dots,\psi(a_g),\psi(b_g),\psi(c_1),\dots,\psi(c_n)\}$$
is a generating set of a Fuchsian group of signature~$(g,n)$.
The Fricke space~$\tT_{g,n}$ is homeomorphic to~$\r^{6g-3+3n}$.
The group $\Aut_+(\hyp)$ acts on $\tT_{g,n}$ by conjugation.
The {\it Teichm\"uller space\/} is $T_{g,n}=\tT_{g,n}/\Aut_+(\hyp)$.

\begin{thm}
\label{thm-moduli-Klein}
Let $(g,k,\ve)$ be a topological type of a Klein surface.
In the case $\ve=1$ let $n=k$.
In the case $\ve=0$ we choose $n\in\{k+1,\dots,g+1\}$ such that $n\equiv g+1~(\mod2)$.
Let $\tilde g=(g+1-n)/2$.
The moduli space~$\calM_{g,k,\ve}$ of Klein surfaces of topological type~$(g,k,\ve)$
is the quotient of the Teichm\"uller space~$T_{\tilde g,n}$ by a discrete group of autohomeomorphisms $\Mod_{g,k,\ve}$.
The space $T_{\tilde g,n}$ is homeomorphic to $\r^{3g-3}$.
%$\r^{6\tilde g-6+3n}=\r^{3g-3}$
\end{thm}

\begin{thm}
The moduli space of Klein surfaces of genus~$g$ decomposes into connected components~$\calM_{g,k,\ve}$.
Each connected component is homeomorphic to a quotient of~$\r^{3g-3}$ by a discrete group action.
\end{thm}

%\subsection{Moduli Spaces of Higher Spin Bundles on Klein Surfaces}

\begin{thm}
\label{thm-moduli-spin}
Let $(g,k,\ve)$ be a topological type of a Klein surface.
Assume that the geometric genus of such Klein surfaces is positive, % LIMIT
i.e.\ $k\le g-2$ if~$\ve=0$ and $k\le g-1$ if $\ve=1$. % LIMIT
Let $t$ be a tuple that satisfies the conditions of Proposition~\ref{topinv-nec}.
The space $S(t)$ of all $m$-spin bundles of type~$t$ on a Klein surface of type $(g,k,\ve)$
is connected and diffeomorphic to
$$\r^{3g-3}/\Mod_t,$$
where $\Mod_t$ is a discrete group of diffeomorphisms.
\end{thm}

\begin{proof}
In the case $\ve=1$ let $n=k$.
In the case $\ve=0$ we choose $n\in\{k+1,\dots,g-1\}$ such that $n\equiv g-1~(\mod2)$.
Let $\tilde g=(g+1-n)/2$.
By definition, to any $\psi\in\tilde T_{\tilde g,n}$ corresponds a generating set 
$$V=\{\psi(a_1),\psi(b_1),\dots,\psi(a_{\tilde g}),\psi(b_{\tilde g}),\psi(c_1),\dots,\psi(c_n)\}$$
%$$V=(A_1,B_1,\dots,A_{\tilde g},B_{\tilde g},C_1,\dots,C_n)
%=\{\psi(a_1),\psi(b_1),\dots,\psi(a_{\tilde g}),\psi(b_{\tilde g}),\psi(c_1),\dots,\psi(c_n)\}$$
of a Fuchsian group of signature~$(\tilde g,n)$.
The generating set~$V$ together with
$$\{\overline{\psi(c_1)},\dots,\overline{\psi(c_k)},\widetilde{\psi(c_{k+1})},\dots,\widetilde{\psi(c_n)}\}$$
%$$(\bar C_1,\dots,\bar C_k,\tilde C_{k+1},\dots,\tilde C_n)$$
generates a real Fuchsian group~$\Ga_{\psi}$.
On the Klein surface $(P,\tau)=[\Ga_{\psi}]$, we consider the corresponding symmetric generating set
$$
  \calB_{\psi}
  =(a_1,b_1,\dots,a_{\tilde g},b_{\tilde g},a_1',b_1',\dots,a_{\tilde g}',b_{\tilde g}',c_1,\dots,c_{n-1},d_1,\dots,d_{n-1}).
$$
%that corresponds to the shifts
%$$
%   &A_1,B_1,\dots,A_{\tilde g},B_{\tilde g},\\
%   &\bar C_k A_1\bar C_k,\bar C_k B_1\bar C_k,\dots,\bar C_k A_{\tilde g}\bar C_k,\bar C_k B_{\tilde g}\bar C_k,\\
%   C_1,\tilde C_{n}\bar C_1,\dots,C_k,\tilde C_{n}\bar C_k,
%   C_{k+1},\tilde C_{n}\tilde C_{k+1},\dots,C_{n-1},\tilde C_{n}\tilde C_{n-1}
%$$
%in the case $\ve=0$ and to the shifts
%\begin{align*}
%  &A_1,B_1,\dots,A_{\tilde g},B_{\tilde g},\\
%  &\bar C_k A_1\bar C_k,\bar C_k B_1\bar C_k,\dots,\bar C_k A_{\tilde g}\bar C_k,\bar C_k B_{\tilde g}\bar C_k,\\
%  &C_1,\bar C_k\bar C_1,\dots,C_{k-1},\bar C_k\bar C_{k-1}
%\end{align*}
%in the case $\ve=1$.
Proposition~\ref{prop-canon2} implies that there exists a real $m$-Arf function $\arf=\arf_{\psi}$ of type~$t$
for which $\calB_{\psi}$ is canonical.
According to Theorem~3{.}11 in~\cite{NP:2016},
an $m$-spin bundle $\Om(\psi)\in S(t)$ is associated with this $m$-Arf function.
The correspondence $\psi\mapsto\Om(\psi)$ induces a map $\Om:T_{\tilde g,n}\to S(t)$.
Let us prove that $\Om(T_{\tilde g,n})=S(t)$.
Indeed, by Theorem~\ref{thm-moduli-Klein}, the map
$$\Psi=\Phi\circ\Om:T_{\tilde g,n}\to S(t)\to\calM_{g,k,\ve},$$
where $\Phi$ is the natural projection, satisfies the condition
$$\Psi(T_{\tilde g,n})=\calM_{g,k,\ve}.$$
The fibre of the map $\Psi$ is represented by the group~$\Mod_{g,k,\ve}$
of all self-homeo\-mor\-phisms of the Klein surface $(P,\tau)$.
By Proposition~\ref{topeq}, this group acts transitively on the set of all real $m$-Arf functions of type~$t$
and hence, by Theorem~3{.}11 in~\cite{NP:2016}, transitively on the fibres $\Phi^{-1}((P,\tau))$.
Thus
$$\Om(T_{\tilde g,n})=S(t)=T_{\tilde g,n}/\Mod_t,\quad\text{where}~\Mod_t\subset\Mod_{g,k,\ve}$$
According to Theorem~\ref{thm-moduli-Klein}, the space~$T_{\tilde g,n}$ is diffeomorphic to~$\r^{3g-3}$.
\end{proof}

%\subsection{Branching Indices of Moduli Spaces}
%\label{sec-branching}

\begin{thm}
\label{thm-moduli-branching}
% nonsep: Nbook, Thm 3.2 on p.80
% sep: Nbook, Thm 3.4 on p.81
Let $(g,k,\ve)$ be a topological type of a Klein surface.
Assume that the geometric genus of such Klein surfaces is positive, % LIMIT
i.e.\ $k\le g-2$ if~$\ve=0$ and $k\le g-1$ if $\ve=1$. % LIMIT
Let $t$ be a tuple that satisfies the conditions of Proposition~\ref{topinv-nec}.
The space $S(t)$ of all real $m$-spin bundles of type~$t$ on a Klein surface of type $(g,k,\ve)$
is an $N(t)$-fold covering of $\calM_{g,k,\ve}$,
where $N(t)$ is the number of real $m$-Arf functions on $(P,\tau)$ of topological type $t$.
The number $N(t)$ is equal to
\begin{enumerate}[1)]
\item
Case $\ve=0$, $m\equiv0~(\mod2)$, $t=(g,\de,k_0,k_1)$:
$$N(t)={k\choose k_1}\cdot\frac{m^g}2.$$ 
% k=k_0+k_1, hence {k\choose k_1}={k\choose k_0}
\item
Case $\ve=1$, $m\equiv0~(\mod2)$, $t=(g,\tilde\de,k^0_0,k^0_1,k^1_0,k^1_1)$:
% k=k_0+k_1, hence {k\choose k_0}={k\choose k_1}
% k_0=k_0^0+k_0^1, hence {k_0\choose k_0^0}={k_0\choose k_0^1}
% k_1=k_1^0+k_1^1, hence {k_1\choose k_1^0}={k_1\choose k_1^1}
Let
$$M={k\choose k_0}\cdot{k_0\choose k_0^0}\cdot{k_1\choose k_1^0}.$$
\begin{enumerate}[$\bullet$]
\item
Case $g>k+1$, ($m\equiv0~(\mod4)$ or $k_0\ne0$):
$$
  N(t)=2^{1-k}\cdot m^g\cdot M\quad\text{for}~\tilde\de=0
  ~\text{and}~
  N(t)=0\quad\text{for}~\tilde\de=1.
$$
\item
Case $g>k+1$, $m\equiv2~(\mod4)$, $k_0=0$:
\begin{align*}
  N(t)&=\left(2^{-k}+2^{-\frac{g+k+1}2}\right)\cdot m^g\cdot M\quad\text{for}~\tilde\de=0,\\
  N(t)&=\left(2^{-k}-2^{-\frac{g+k+1}2}\right)\cdot m^g\cdot M\quad\text{for}~\tilde\de=1.
\end{align*}
\item
Case $g=k+1$, ($m\equiv0~(\mod4)$ or $k_0\ne0$):
$$
  N(t)=2^{-(k-1)}\cdot m^{k+1}\cdot M\quad\text{for}~\tilde\de=1
  ~\text{and}~
  N(t)=0\quad\text{for}~\tilde\de=2.
$$
\item
Case $g=k+1$, $m\equiv2~(\mod4)$, $k_0=0$:
\begin{align*}
  N(t)&=3\cdot 2^{-(k+1)}\cdot m^{k+1}\cdot M\quad\text{for}~\tilde\de=1,\\
  N(t)&=2^{-(k+1)}\cdot m^{k+1}\cdot M\quad\text{for}~\tilde\de=2.
\end{align*}
\end{enumerate}
\item
Case $m\equiv1~(\mod2)$, $t=(g,k)$:
$$N(t)=m^g.$$
\end{enumerate}
\end{thm}

\begin{proof}
According to Theorem~\ref{thm-moduli-spin}, $S(t)\cong T_{\tilde g,n}/\Mod_t$, where $\Mod_t\subset\Mod_{g,k,\ve}$,
hence $S(t)$ is a branched covering of $\calM_{g,k,\ve}=T_{\tilde g,n}/\Mod_{g,k,\ve}$
and the branching index is equal to the index of the subgroup $\Mod_t$ in $\Mod_{g,k,\ve}$,
i.e.\ is equal to the number~$N(t)$ of real $m$-Arf functions on $(P,\tau)$ of topological type $t$.
Let
$$
  \calB
  =(a_1,b_1,\dots,a_{\tilde g},b_{\tilde g},a_1',b_1',\dots,a_{\tilde g}',b_{\tilde g}',c_1,d_1,\dots,c_{n-1},d_{n-1})
$$
be a symmetric generating set of~$\piorb(P)$.
Let $\calV=(\al_i,\be_i,\al_i',\be_i',\ga_i,\de_i)$ denote the set of values of an $m$-Arf function on~$\calB$.
\begin{enumerate}[1)]
\item
Case $\ve=0$, $m\equiv0~(\mod2)$, $t=(g,\de,k_0,k_1)$:
There are ${k\choose k_1}$ ways to choose the values $\ga_i$.
There are $m^{2\tilde g}$ ways to choose $\al_i=\al_i'$ and $\be_i=\be_i'$.
According to Theorem~4{.}9(5) in~\cite{NP:2016},
out of $m^{n-1}$ ways to choose $\de_1,\dots,\de_{n-1}$
there are $m^{n-1}/2$ which give $\Si\equiv0~(\mod2)$ and $m^{n-1}/2$ which give $\Si\equiv1~(\mod2)$.
Thus the number of real $m$-Arf functions of type $(g,\de,k_0,k_1)$ is
$$
  {k\choose k_1}\cdot m^{2\tilde g}\cdot\frac{m^{n-1}}2
  ={k\choose k_1}\cdot\frac{m^{2\tilde g+n-1}}2
  ={k\choose k_1}\cdot\frac{m^g}2.
$$
\item
Case $\ve=1$, $m\equiv0~(\mod2)$, $t=(g,\tilde\de,k^0_0,k^0_1,k^1_0,k^1_1)$:
There are $M={k\choose k_0}\cdot{k_0\choose k_0^0}\cdot{k_1\choose k_1^0}$ ways
to choose the values $\ga_i$.
Furthermore having fixed the parity of $\de_i$,
there are $(m/2)^{k-1}$ ways to choose the values of $\de_i$.
Hence the number of such real $m$-Arf functions on~$P$ is equal to 
$$
  m^{2\tilde g}\cdot\left(\frac{m}2\right)^{k-1}\cdot M
  =\frac{m^{2\tilde g+k-1}}{2^{k-1}}\cdot M
  =m^g\cdot 2^{1-k}\cdot M.
$$
\begin{enumerate}[$\bullet$]
\item
In the case $g>k+1$, $m\equiv2~(\mod4)$, $k_0=0$,
the resulting invariant~$\tilde\de$ is given by
$$\tilde\de\equiv\sum\limits_{i=1}^{\tilde g}\,(1-\al_i)(1-\be_i)~(\mod2).$$
It can be shown by induction that out of $m^{2\tilde g}$ ways to choose the values $\al_i$, $\be_i$
we get the Arf invariant $\tilde\de=0$ in $2^{\tilde g-1}(2^{\tilde g}+1)(m/2)^{2\tilde g}$ cases
and $\tilde\de=1$ in $2^{\tilde g-1}(2^{\tilde g}-1)(m/2)^{2\tilde g}$ cases.
Hence the number $N(t)$ with $\tilde\de$ equal to~$0$ and~$1$ respectively is
$$2^{\tilde g-1}(2^{\tilde g}\pm1)\left(\frac{m}2\right)^{2\tilde g}\left(\frac{m}2\right)^{k-1}\cdot M.$$
We simplify
%2\tilde g=g+1-k
\begin{align*}
  &2^{\tilde g-1}(2^{\tilde g}\pm1)\left(\frac{m}2\right)^{2\tilde g}\left(\frac{m}2\right)^{k-1}
  =(2^{2\tilde g-1}\pm2^{\tilde g-1})\left(\frac{m}2\right)^{2\tilde g+k-1}\\
  &=\left(2^{g-k}\pm2^{\frac{g-k-1}2}\right)\left(\frac{m}2\right)^{g}
  =\left(2^{g-k}\pm2^{\frac{g-k-1}2}\right)2^{-g}\cdot m^g\\
  &=\left(2^{-k}\pm2^{\frac{-g-k-1}2}\right)m^g
  =\left(2^{-k}\pm2^{-\frac{g+k+1}2}\right)m^g
\end{align*}
to obtain $N(t)$ as stated.
\item
In the case $g>k+1$, ($m\equiv0~(\mod4)$ or $k_0\ne0$),
the Arf invariant of all $m$-Arf functions we construct is $\tilde\de=0$, hence $N(t)$ is as stated.
\item
In the case $g=k+1$, $m\equiv2~(\mod4)$, $k_0=0$,
the Arf invariant of the resulting $m$-Arf function is given by
$$\tilde\de=\gcd\left(m,\al_1,\be_1,\frac{m}2+1\right).$$
Note that for $m\equiv2~(\mod4)$ we have $\gcd(m,m/2+1)=2$,
hence $\tilde\de=2$ if $\al_1$ and $\be_1$ are both even and $\tilde\de=1$ otherwise.
Out of $m^2$ ways to choose the values $\al_1$, $\be_1$
we get $\tilde\de=1$ in $3m^2/4$ cases and $\tilde\de=2$ in $m^2/4$ cases.
Hence the number $N(t)$ with $\tilde\de$ equal to~$1$ and~$2$ respectively is
$$\frac{2\pm1}4\cdot m^2\left(\frac{m}2\right)^{k-1}\cdot M=(2\pm1)\cdot\left(\frac{m}2\right)^{k+1}\cdot M.$$
\item
In the case $g=k+1$, ($m\equiv0~(\mod4)$ or $k_0\ne0$),
the Arf invariant of all $m$-Arf functions we construct is $\tilde\de=1$,
hence $N(t)$ is as stated.
\end{enumerate}
\item
Case $m\equiv1~(\mod2)$, $t=(g,k)$:
The statement follows from Theorem~4{.}10(3) in~\cite{NP:2016}.
\end{enumerate}
\end{proof}

\myskip\noindent
{\bf Example:}
Consider the case $g=3$, $m=4$.
Let $P$ be a compact Riemann surface of genus~$3$.
According to Weichold's classification for a Klein surface $(P,\tau)$ either $\ve=1$, $k\in\{2,4\}$ or $\ve=0$, $k\in\{0,1,2,3\}$.
Possible topological types of $4$-spin bundles on these Klein surfaces
are described in Propositions~\ref{topinv-nec} and~\ref{topinv-suff}.
Condition $k_1\cdot m/2\equiv1-g~(\mod m)$ becomes $2k_1\equiv-2~(\mod4)$ and is equivalent to $k_1$ being odd.

\myskip
For example there exist $4$-spin bundles on separating Klein surfaces $(P,\tau)$ with $k=2$
and for these bundles $k_0^0+k_0^1=k_1^0+k_1^1=1$ and $\tilde\de=1$,
i.e.\ the bundle is trivial on one of the ovals and non-trivial on the other and the $4$-spin bundle restricted to $P\backslash P^{\tau}$ is odd.
There are two possible topological types of such bundles up to the swap $k_j^i\leftrightarrow k_j^{1-i}$: 
$$
  (g,\tilde\de,k_0^0,k_1^0,k_0^1,k_1^1)=(3,1,1,1,0,0)
  \quad\text{and}\quad
  (g,\tilde\de,k_0^0,k_1^0,k_0^1,k_1^1)=(3,1,1,0,0,1).
$$
Proposition~\ref{topeq} implies that $4$-spin bundles on separating Klein surfaces of genus~$g=3$ with $k=2$
are topologically equivalent if and only if they have the same topological type.
Theorem~\ref{thm-moduli-branching} implies that the number~$N(t)$ of real $4$-spin bundles
of topological type $t=(g,\tilde\de,k_0^0,k_1^0,k_0^1,k_1^1)$
is $N(t)=64$ for $t=(3,1,1,1,0,0)$ and $t=(3,1,1,0,0,1)$.
%N(t)=2^{-(k-1)}\cdot m^{k+1}\cdot{k\choose k_0}\cdot{k_0\choose k_0^0}\cdot{k_1\choose k_1^0}
% t=(g,\tilde\de,k_0^0,k_1^0,k_0^1,k_1^1)=(3,1,1,1,0,0)
% => N(t)=2^{-1}\cdot 4^3\cdot{2\choose 1}\cdot{1\choose 1}\cdot{1\choose 1}
% t=(g,\tilde\de,k_0^0,k_1^0,k_0^1,k_1^1)=(3,1,1,0,0,1)
%  => N(t)=2^{-1}\cdot 4^3\cdot{2\choose 1}\cdot{1\choose 1}\cdot{1\choose 0}

\myskip
There exist $4$-spin bundles on non-separating Klein surfaces $(P,\tau)$ with $k=1$
and for these bundles $k_0=0$, $k_1=1$, i.e.\ the bundle is non-trivial on the only oval.
There are two possible topological types of such bundles:
$$
  (g,\de,k_0,k_1)=(3,0,0,1) 
  \quad\text{and}\quad
  (g,\de,k_0,k_1)=(3,1,0,1).
$$
Proposition~\ref{topeq} implies that $4$-spin bundles on non-separating Klein surfaces of genus~$g=3$ with $k=1$
are topologically equivalent if and only if they have the same topological type.
Theorem~\ref{thm-moduli-branching} implies that the number~$N(t)$ of real $4$-spin bundles
of topological type $t=(g,\de,k_0,k_1)$ is $N(t)=32$ for $t=(3,0,0,1)$ and $t=(3,1,0,1)$.
% N(t)={k\choose k_1}\cdot\frac{m^g}{2}={1\choose 1}\cdot\frac{4^3}{2}

\myskip
Similarly separating Klein surfaces $(P,\tau)$ with $k=4$
admit $4$-spin bundles with topological types $(g=3,\tilde\de,k_0^0,k_1^0,k_0^1,k_1^1)$
with  $(k_0,k_1)=(k_0^0+k_0^1,k_1^0+k_1^1)=(1,3), (3,1)$
and non-separating Klein surfaces $(P,\tau)$ with $k=2,3$ admit $4$-spin bundles
of topological types $(g=3,\tilde\de,k_0^0,k_1^0,k_0^1,k_1^1)$ with
$(k_0^0+k_0^1,k_1^0+k_1^1)=(1,1), (0,3), (2,1)$,
while non-separating Klein surfaces $(P,\tau)$ with $k=0$ do not admit $4$-spin structures.
Geometric genus of Klein surfaces with $\ve=1$, $k=4$ and $\ve=0$, $k=2,3$ is equal to zero
and their topological equivalence is not considered in this paper.

\section{Applications in Singularity Theory}

\label{sec-sing}

\myskip
I.~Dolgachev in~\cite{Dolgachev:1975, Dolgachev:1977, Dolgachev:1983} described how all hyperbolic Gorenstein quasi-homogeneous surface singularities can be constructed by contracting the zero section of an $m$-spin bundle on $\hyp/\Ga$ for some Fuchsian group~$\Ga$.
(If the group~$\Ga$ has torsion, a more careful construction using a normal torsion-free subgroup of~$\Ga$ of finite index is necessary.)
Hence a Klein surface structure on~$\hyp/\Ga$ leads to an anti-holomorphic involution on the singularity, i.e.\ to a real form of the singularity.

\myskip
The correspondence between the weights of a quasi-homogeneous singularity and the signature of the Fuchsian group was studied in detail by K.~M\"ohring in~\cite{Moehring:dipl, Moehring:paper}.
In this paper we only consider the case where $\Ga$ is a surface group,
i.e.\ a Fuchsian group such that $\hyp/\Ga$ is a compact Riemann surface.
For general Gorenstein quasi-homogeneous surface singularities we need to consider Fuchsian groups $\Ga$ with torsion and $m$-spin bundles on the corresponding Klein orbifolds, i.e.\ on orbifolds $\hyp/\Ga$ with an anti-holomorphic involution.
The first results in this direction were obtained by Riley~\cite{Ril} who considered the case where the marked points of the orbifold $\hyp/\Ga$ do not lie on the set of real points~$P^{\tau}$.

\myskip
Let $\Ga$ be a Fuchsian group such that $\hyp/\Ga$ is a compact Riemann surface of genus~$g$.
Let $W$ be a corresponding weight system for a quasi-homogeneous singularity
as described in~\cite{Moehring:dipl, Moehring:paper}.
M\"ohring states (Example 2{.}7 in~\cite{Moehring:paper}) that among all quasi-homogeneous hypersurface singularities with the weight system~$W$ there is always a Brieskorn-Pham singularity, i.e.\ a singularity of the form $x^a+y^b+z^c=0$.
Moreover, a (non-regular) normal quasi-homogeneous hypersurface singularity corresponds to a surface group if and only if it has the same weight system as a Brieskorn-Pham singularity $x^a+y^b+z^c=0$
such that no prime divides only one of the exponents~$a,b,c$ (see section~7 in~\cite{Milnor:1975}).

\myskip
For example $4$-spin bundles on surfaces of genus~$3$ whose real forms were discussed above correspond
to Brieskorn-Pham singularities $x^{14}+y^7+z^2=0$ and $x^{12}+y^4+z^3=0$ (Example~2.3 in~\cite{Moehring:paper}).
It would be of interest to make the connection between the anti-holomorphic involutions on a Riemann surface
and on the corresponding singularities more explicit but this is beyond the scope of this paper.

%%\nocite{*}
 %%\bibliographystyle{amsalpha}
%%\bibliography{realforms}

\def\cprime{$'$}
\providecommand{\bysame}{\leavevmode\hbox to3em{\hrulefill}\thinspace}
\providecommand{\MR}{\relax\ifhmode\unskip\space\fi MR }
% \MRhref is called by the amsart/book/proc definition of \MR.
\providecommand{\MRhref}[2]{%
  \href{http://www.ams.org/mathscinet-getitem?mr=#1}{#2}
}
\providecommand{\href}[2]{#2}

\end{document}